\documentclass{SCAEOL}
\numberwithin{equation}{section}
\usepackage{amssymb}
\usepackage{mathrsfs}
\usepackage{hyperref}
\usepackage{color}
\usepackage{graphicx}
\usepackage{epstopdf}
\usepackage{mathtools}
\usepackage{graphicx}
\usepackage{subfigure}

\numberwithin{equation}{section}
\begin{document}

\Year{2020} %
\Month{September}
\Vol{56} %
\No{1} %
\BeginPage{1} %
\EndPage{XX} %
\AuthorMark{Yongjian Wang {\it et al.}}

\title[Coexistence of zero LE and positive LE]{Coexistence of zero Lyapunov exponent and positive Lyapunov exponent for new quasi-periodic Schr$\ddot{o}$dinger operator}{}


\author[1,2]{Yongjian Wang}{}
\author[1,2,3]{Zuohuan Zheng}{Corresponding author}

%
\address[{\rm1}]{Academy of Mathematics and Systems Science, Chinese Academy of Sciences, Beijing {\rm 100190}, China;}
\address[{\rm2}]{University of Chinese Academy of Sciences, Beijing {\rm 100049}, China;}
\address[{\rm3}]{College of Mathematics and Statistics, Hainan Normal University, Haikou, Hainan {\rm571158}, China;}
\Emails{wangyongjian@amss.ac.cn,zhzheng@amt.ac.cn}\maketitle


 {\begin{center}
\parbox{14.5cm}{\begin{abstract}
In this paper we solve a problem about the Schr$\ddot{o}$dinger operator with potential $v(\theta)=2\lambda cos2\pi\theta/(1-\alpha cos2\pi\theta),\ (|\alpha|<1)$ in physics. With the help of the formula of Lyapunov exponent in the spectrum, the coexistence of zero Lyapunov exponent and positive Lyapunov exponent for some parameters is first proved, and there exists a curve that separates them. The spectrum in the region of positive Lyapunov exponent is purely pure point spectrum with exponentially decaying eigenfunctions for almost every frequency and almost every phase. From the research, we realize that the infinite potential $v(\theta)=2\lambda tan^2(\pi\theta)$ has zero Lyapunov exponent for some energies if $0<|\lambda|<1$.\vspace{-3mm}
\end{abstract}}\end{center}}



\renewcommand{\baselinestretch}{1.2}
\begin{center} \renewcommand{\arraystretch}{1.5}
{\begin{tabular}{lp{0.8\textwidth}} \hline \scriptsize
{\bf Citation:}\!\!\!\!&\scriptsize First1 L N, First2 L N, First3 L N.  SCIENCE CHINA Mathematics  journal sample. Sci China Math, 2013, 56, doi: 10.1007/s11425-000-0000-0\vspace{1mm}
\\
\hline
\end{tabular}}\end{center}

\baselineskip 11pt\parindent=10.8pt  \wuhao
\section{Introduction and main results}
Since the 1970's, Schr$\ddot{o}$dinger operators have been popular in solid-state physics. Schr$\ddot{o}$dinger operators with random or quasi-periodic potentials can describe the influence of an external magnetic field on the electrons of a crystal, and model Hamiltonians of quantum mechanical systems. Anderson model and almost Mathieu operator are two widely studied examples. The spectral type including absolutely continuous spectrum, singular continuous spectrum and pure point spectrum is one of the major researches.

In physics, Anderson localization \cite{APW} describes insulating behavior in the sense that quantum states are localized in a bounded region all the time. The quantum state satisfying Anderson localization is called localized state, otherwise is called extended state. In mathematics, Anderson localization means pure point spectrum with exponentially decaying eigenfunctions.

The almost Mathieu operator ($v(\theta)=2\lambda cos2\pi\theta$) has been throughly studied and has purely spectral types in three different cases \cite{J}, \cite{AYZ}: for almsot every pair $(\theta,b)$ the Schr$\ddot{o}$dinger operator has purely absolutely continuous spectrum if $|\lambda|<1$, purely pure point spectrum with exponentially decaying eigenfunctions if $|\lambda|>1$ and purely singular continuous spectrum if $|\lambda|=1$. Therefor, phase transition occurs when $\lambda$ goes from $|\lambda|>1$ to $|\lambda|<1$. However, we don't know what is happening in the transition region ($\lambda\approx1$). In fact, we have a deep understanding of the nature of Sch$\ddot{o}$dinger operators with ``large''($|\lambda|\gg1$) and ``small''($|\lambda|\ll1$) analytic potentials: for ``small'' $|\lambda|$ the Schr$\ddot{o}$dinger operator has purely absolutely continuous spectrum and has zero Lyapunov exponent in the spectrum\cite{BoJ}; for ``large'' $|\lambda|$ the Schr$\ddot{o}$dinger operator has pure point spectrum (for almost every $\theta$) with exponentially decaying eigenfunctions and positive Lyapunov exponent for all energies in $\mathbb{R}$\cite{BG}. Nonetheless, the phase transition between absolutely continuous and pure point spectrum has been considerably harder to understand. Hence it is essential to study mixed spectral types for Schr$\ddot{o}$dinger operators though such examples for one-frequency discrete case have been considered difficult to construct explicitly. Bourgain \cite{JB} constructed a quasi-periodic Schr$\ddot{o}$dinger operator with two frequencies which has both absolutely continuous and pure point spectrum. Bjerkl$\ddot{o}$v \cite{K} gave examples which have positive Lyapunov exponent on certain regions of the spectrum and zero on other regions. In his examples, operators with arbitrarily large potentials may have zero Lyapunov exponent for certain energies. However, he failed to prove the coexistence of more than two spectral types. Zhang \cite{Z}proved these examples have coexistence of absolutely continuous and pure point spectrum for some parameters as well as coexistence of absolutely continuous and singular continuous spectrum  for some other parameters. Avila \cite{A} showed that perturbations of the critical almost Mathieu operator (with potential $v(\theta)=2cos2\pi\theta$) may have arbitrarily many alternances between subcritical and supercritical regimes.

We consider the following one-dimensional quasi-periodic Schr$\ddot{o}$dinger operator on $l^2$($\mathbb{Z}$):

\begin{equation}\label{mymodel}
	(H_{\alpha,\theta,b}u)_{n}=u_{n+1}+u_{n-1}+v(nb+\theta)u_{n},\ n\in \mathbb{Z},
\end{equation}
where
\begin{equation}\label{mypotential}
	v(\theta)=2\lambda\frac{cos(2\pi\theta)}{1-\alpha cos(2\pi\theta)},\ \alpha\in (-1,1),
\end{equation}
$v$ is the potential, $\alpha$ is the 
parameter, $\theta\in\mathbb{T}=\mathbb{R}/\mathbb{Z}$ is the phase, $\lambda\in\mathbb{R}$ is the coupling, the frequency $b\in\mathbb{R}$ is irrational. It is trivial when $\lambda=0$.

This model was first introduced by Ganeshan S, Pixley J H and Sarma S D in the top journals (Physical Review Letters) \cite{GHD}. They did some calculations based on self-dual condition and did some realistic experiments about atomic optical lattices and photonic waveguides to numerically verify the following problem:
\begin{problem}\label{problem1.1}\quad
For the model \ref{mymodel}, there exists a one-dimensional mobility edge 
\begin{equation}\label{edge}
\alpha E=2sgn(\lambda)(1-|\lambda|),
\end{equation}
that separates localized states from extended states if they coexist. 
\end{problem}
See Figure \ref{figure} for illustration.
\begin{figure}[htbp]
\centering
\subfigure[$\lambda=-0.9$]{
\begin{minipage}[t]{0.5\linewidth}
\centering
\includegraphics[scale=0.68]{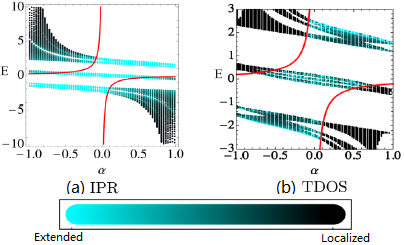}
\end{minipage}%
}%
\subfigure[$\lambda=-1.1$]{
\begin{minipage}[t]{0.5\linewidth}
\centering
\includegraphics[scale=0.68]{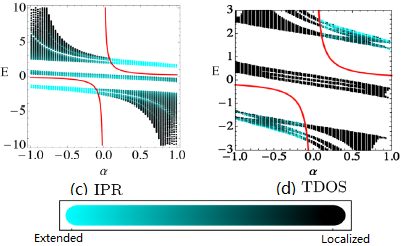}
\end{minipage}%
}%
\centering
\caption{Relations between $E$ in the spectrum and parameter $\alpha$ for different $\lambda$}\label{figure}
\end{figure}
The x-coordinate represents the parameter $\alpha$, and the y-coordinate represents the corresponding energy in the spectrum.
The localization properties of a quantum state can be numerically quantified by $IPR$ and $TDOS$ (see the definitions in \cite{GHD}). Pure cyan denotes $IPR=0$ for the extended state and pure black denotes $IPR=1$ for the localized state. In addition, pure cyan denotes maximum $TDOS$ values between 1 and 10 for the extended state and pure black denotes $TDOS=0$ for the localized state. It is clear that the red curve (\ref{edge}) separates localized states from extended states.

In this paper, we give the first strict proof of the Problem \ref{problem1.1}. In addition, we show that zero Lyapunov exponent and positive Lyapunov exponent coexist under some conditions. Moreover, we give an example that has infinite potential and zero Lyapunov exponent for some energies.

Our main results of this paper are the following theorems:
\begin{theorem}\label{theorem2.8}\quad For the model (\ref{mymodel}), we have the formula of Lyapunov exponent 
	$$L(b,A)=max\{log|\frac{\alpha E+2\lambda\pm\sqrt{(\alpha E+2\lambda)^2-4\alpha^2}}{2(1+\sqrt{1-\alpha^2})}|,0\}$$
	for every $E$ in the spectrum.
\end{theorem}

\begin{theorem}\label{theorem3.2}\quad In the model (\ref{mymodel}), for every $E$ in the spectrum,

	\begin{equation}\label{poly}L(b,A)>0\Longleftrightarrow\left\{\quad\begin{matrix*}[l]\alpha E>2sgn(\lambda)(1-|\lambda|)&\lambda>0,\\ \alpha E<2sgn(\lambda)(1-|\lambda|)&\lambda<0,\end{matrix*}\right.
	\end{equation}
and	
	\begin{equation}\label{zely}L(b,A)=0\Longleftrightarrow\left\{\quad\begin{matrix*}[l]\alpha E\le2sgn(\lambda)(1-|\lambda|)&\lambda>0,\\ \alpha E\ge2sgn(\lambda)(1-|\lambda|)&\lambda<0,\\ \forall\,E\in\sum\nolimits_{b,v}&
	\lambda=0.\end{matrix*}\right.
	\end{equation}
	In particular, when $\lambda\neq0$ the curve
	\begin{equation}
	\alpha E=2sgn(\lambda)(1-|\lambda|)
	\end{equation}
	separates the spectrum with positive Lyapunov exponent from the spectrum with zero Lyapunov exponent if they coexist.
\end{theorem}

\begin{theorem}\label{theorem3.3}\quad 
In the model (\ref{mymodel}), there are both positive Lyapunov exponent and zero Lyapunov exponent in the spectrum if $1-|\alpha|<|\lambda|<1+|\alpha|$; there is only positive Lyapunov exponent in the spectrum if $|\lambda|>(1+|\alpha|)^2$; there is only zero Lyapunov exponent in the spectrum if $|\lambda|\le(1-|\alpha|)^2$.
\end{theorem}
\begin{corollary}\label{corollary1.6}\quad
Consider the potential $v(\theta)=2\lambda tan^2\pi\theta$, there are both zero Lyapunov exponent and positive Lyapunov exponent in the spectrum if $0<|\lambda|<1$.
\end{corollary}
This potential is infinite and has a singular point in $\mathbb{T}$.
\begin{theorem}\label{theorem3.17}\quad 
Suppose $b\in\mathbb{T}$ is Diophantine and $\theta\in\mathbb{T}$ is non-resonant. For the model (\ref{mymodel}), spectrum in the region of positive Lyapunov exponent is purely pure point spectrum with exponentially decaying eigenfunctions (Anderson localization).
\end{theorem}

It is known that Lebesgue almost every $b\in\mathbb{T}$ is Diophantine and Lebesgue almost every $\theta\in\mathbb{T}$ is non-resonant. Thus the theorem holds for almost every frequency and almost every phase.

\section{The formula of Lyapunov exponent}
In this section, we use Herman's subharmonicity methods and Avila's global theroy to obtain the formula of Lyapunov exponent in the spectrum.

Given a bounded map $v:\mathbb{Z}\to\mathbb{R}$, the solutions of Schr$\ddot{o}$dinger equation 
\begin{equation}\label{equation2.1}
	(H_{\alpha,\theta,b}u)_{n}=u_{n+1}+u_{n-1}+v(nb+\theta)u_{n}=zu_{n},
\end{equation}
which, for $n\ge 1$, can be expressed as
\begin{equation}
	\begin{pmatrix}u_{n+1}\\u_{n}\end{pmatrix}=A_{n}(\theta,z)\begin{pmatrix}u_{1}\\u_{0}\end{pmatrix},
\end{equation}
\begin{equation*}
	A_{n}(\theta,z)=A(\theta+(n-1)b)\dots A(\theta),\ A_{-n}(\theta,z)=A_{n}
	(\theta-nb)^{-1},\ n>0.
\end{equation*}
\begin{equation}
	A(\theta,z)=A^{(z-v)}(\theta)=\begin{pmatrix}z-v(\theta)&-1\\1&0\end{pmatrix},
	\ \theta\in\mathbb{T},\ z\in\mathbb{Z}.
\end{equation}
The spectrum $\sigma(H_{\alpha,\theta,b})$ of the operator $H_{\alpha,\theta,b}$ is a nonempty, compact subset of $\mathbb{R}$. When $b$ is irrational it is the same for almost every $\theta\in\mathbb{T}$, we denote it by $\sum\nolimits_{b,v}$.
The Lyapunov exponent as usual is defined by

\begin{equation}\label{equation2.4}
	L(z)=L(b,A^{(z-v)})=\mathop{lim}\limits_{n\rightarrow\infty}\int_{\mathbb{T}}\frac{1}{n}log
	\parallel A_{n}(\theta,z)\parallel d\theta.
\end{equation}
The Lyapunov exponent defined above is non-negative since $det(A_{n})=1$ for every $n\in\mathbb{Z}$.

Now, we use Herman's subharmonicity argument and extensions to estimate a lower bound of the Lyapunov exponent of the model (\ref{mymodel}).
\begin{theorem}\label{theorem2.1}\quad For the model (\ref{mymodel}), the Lyapunov exponent has a lower bound at all energies $z\in\mathbb{C}$:
	$$L(z)\ge max\{log|\frac{\alpha z+2\lambda\pm\sqrt{(\alpha z+2\lambda)^2-4\alpha^2}}{2(1+\sqrt{1-\alpha^2})}|,0\}$$
\end{theorem}

\begin{proof}\quad
	Note that
	\begin{equation*}
		\begin{split}
			A(\theta,z) &=\begin{pmatrix}z-2\lambda\frac{cos(2\pi\theta)}{1-\alpha cos(2\pi\theta)}&-1\\1&0\end{pmatrix} \\
			&=\frac{1}{2-\alpha e^{2\pi\theta i}-\alpha e^{-2\pi\theta i}}\begin{pmatrix}2z-(\alpha z+2\lambda)(e^{2\pi\theta i}+e^{-2\pi\theta i})&-2+\alpha e^{2\pi\theta i}+\alpha e^{-2\pi\theta i}\\2-\alpha e^{2\pi\theta i}-\alpha e^{-2\pi\theta i}&0\end{pmatrix}.
		\end{split}
	\end{equation*}
Denote
\begin{equation*}
\begin{split}
B(\theta,z)&=(2-\alpha e^{2\pi\theta i}-\alpha e^{-2\pi\theta i})A(\theta,z) \\&=\begin{pmatrix}2z-(\alpha z+2\lambda)(e^{2\pi\theta i}+e^{-2\pi\theta i})&-2+\alpha e^{2\pi\theta i}+\alpha e^{-2\pi\theta i}\\2-\alpha e^{2\pi\theta i}-\alpha e^{-2\pi\theta i}&0\end{pmatrix}.
\end{split}
\end{equation*}
Similarly, we can define
 $B_{n}(\theta,z)=B(\theta+(n-1)b)\dots B(\theta)$.
 
According to Birkhoff theorem and Jensen theorem, we have
	$$L(z)=\mathop{lim}\limits_{n\rightarrow\infty}\int_{\mathbb{T}}\frac{1}{n}log\parallel B_{n}(\theta,z)\parallel d\theta-\int_{\mathbb{T}}log|2-\alpha e^{2\pi\theta i}-\alpha e^{-2\pi\theta i}|d\theta,$$
and
	$$\int_{\mathbb{T}}log|2-\alpha e^{2\pi\theta i}-\alpha e^{-2\pi\theta i}|d\theta=log|1+\sqrt{1-\alpha^2}|.$$	
Setting $\omega=e^{2\pi\theta i}$, we find that
\begin{equation*}
\begin{split}
\omega B(\theta,z)&=\begin{pmatrix}2z\omega-(\alpha z+2\lambda)(\omega^2+1)&-2\omega+\alpha\omega^2+\alpha\\2\omega-\alpha\omega^2-\alpha&0\end{pmatrix}
\end{split}.
\end{equation*}
We define a new function
	$$N_{n}(\omega)=\omega^nB_{n}(\theta,z),$$
initially on $|\omega|=1$, then $N_{n}$ extends to an entire function and hence $\omega\mapsto log\parallel N_{n}(\omega)\parallel$ is subharmonic. 
Hence it holds that
	\begin{equation*}
		\begin{split}
			\int_{\mathbb{T}}\frac{1}{n}log\parallel B_{n}(\theta,z)\parallel d\theta&=\int_{\mathbb{T}}\frac{1}{n}log\parallel\omega^{n} B_{n}(\theta,z)\parallel d\theta\\&=\int_{\mathbb{T}}log\parallel N_{n}(\omega)\parallel d\theta \\&\ge log\parallel N_{n}(0)\parallel\\
			&=log\parallel\begin{pmatrix}-\alpha z-2\lambda&\alpha\\-\alpha&0\end{pmatrix}^{n}\parallel.
		\end{split}
	\end{equation*}
Therefor,
	\begin{equation*}
		\begin{split}
			L(z)&=\mathop{lim}\limits_{n\rightarrow\infty}\int_{\mathbb{T}}\frac{1}{n}log\parallel B_{n}(\theta,z)\parallel d\theta-log|1+\sqrt{1-\alpha^2}|\\
			&=\mathop{lim}\limits_{n\rightarrow\infty}\int_{T}\frac{1}{n}log\parallel N_{n}(\omega)\parallel d\theta-log|1+\sqrt{1-\alpha^2}|\\
			&\ge\mathop{lim}\limits_{n\rightarrow\infty}log\parallel\begin{pmatrix}-\alpha z-2\lambda&\alpha\\-\alpha&0\end{pmatrix}^{n}\parallel-log|1+\sqrt{1-\alpha^2}|\\
			&=max\{log|\frac{\alpha z+2\lambda\pm\sqrt{(\alpha z+2\lambda)^2-4\alpha^2}}{2(1+\sqrt{1-\alpha^2})}|\}.
		\end{split}
	\end{equation*}
	Due to the non-negativity of the Lyapunov exponent, we have
	$$L(z)\ge max\{log|\frac{\alpha z+2\lambda\pm\sqrt{(\alpha z+2\lambda)^2-4\alpha^2}}{2(1+\sqrt{1-\alpha^2})}|,0\}.$$
\end{proof} 
Denote
$$A_{\varepsilon}(\theta,z)=A(\theta+\varepsilon i,z),\ \theta \in\mathbb{T},\ z\in\mathbb{C},\ \varepsilon\in\mathbb{R},$$
it is easy to know that 
$2-\alpha e^{2\pi (\theta+\varepsilon i)i}-\alpha e^{-2\pi (\theta+\varepsilon i)i}$ has at most two real roots about $\varepsilon$ and they respectively satisfy
$$e^{2\pi\theta i}=e^{2\pi\varepsilon}\frac{1+\sqrt{1-\alpha^2}}{\alpha}\ or \ e^{2\pi\theta i}=e^{2\pi\varepsilon}\frac{1-\sqrt{1-\alpha^2}}{\alpha}.$$
Take modulus
$$1=e^{2\pi\varepsilon}|\frac{1+\sqrt{1-\alpha^2}}{\alpha}|\ or\ 1=e^{2\pi\varepsilon}|\frac{1-\sqrt{1-\alpha^2}}{\alpha}|.$$
When $|\varepsilon|<\frac{1}{2\pi}log|\frac{1+\sqrt{1-\alpha^2}}{\alpha}|$, we have
$$1<e^{2\pi\varepsilon}|\frac{1+\sqrt{1-\alpha^2}}{\alpha}|\ and \ 1>e^{2\pi\varepsilon}|\frac{1-\sqrt{1-\alpha^2}}{\alpha}|,$$
it means that $2-\alpha e^{2\pi (\theta+\varepsilon i)i}-\alpha e^{-2\pi (\theta+\varepsilon i)i}$ has no real roots about $\varepsilon$ under this condition. Thus, $A(\cdot,z)\in C^{\omega}(\mathbb{R}/\mathbb{Z},SL(2,\mathbb{C}))$ admits a holomorphic extension to $|Im\cdot|<\frac{1}{2\pi}log|\frac{1+\sqrt{1-\alpha^2}}{\alpha}|$ and $A_{\varepsilon}(\cdot,z)\in C^{\omega}(\mathbb{R}/\mathbb{Z},SL(2,\mathbb{C}))$ is well-defined.
Denote $\delta=\frac{1}{2\pi}log|\frac{1+\sqrt{1-\alpha^2}}{\alpha}|$, by Jensen theorem, we have
$$\int_{\mathbb{T}}log|2-\alpha e^{2\pi (\theta+\varepsilon i)i}-\alpha e^{-2\pi (\theta+\varepsilon i)i}|d\theta=\begin{matrix*}[l]log|1+\sqrt{1-\alpha^2}|&(if\ |\varepsilon|<\delta)\end{matrix*}.$$
When $|\varepsilon|<\delta$, it holds that
\begin{equation}\label{equation2.5}
	\begin{split}
		L(b,A_{\varepsilon})&=L(b,B_{\varepsilon})-\int_{\mathbb{T}}log|2-\alpha e^{2\pi (\theta+\varepsilon i)i}-\alpha e^{-2\pi (\theta+\varepsilon i)i}|d\theta\\
		&=L(b,B_{\varepsilon})-log|1+\sqrt{1-\alpha^2}|.
	\end{split}
\end{equation}

We consider the class of 1-periodic functions on $\mathbb{R}$ which have analytic extension to some strip $|\Im z|<\eta$ and take values in complex 2$\times$ 2 matrices. We denote them by $C^{\omega}(\mathbb{T},M_{2}(\mathbb{C}))$.

Cocycles $(b,D)$ with D $\in C^{\omega}(\mathbb{T},M_{2}(\mathbb{C}))$ are called analytic. For analytic cocycles, we have a similar definition of Lyapunov exponent $L(b,D):\mathbb{T}\times C^{\omega}(\mathbb{T},M_{2}(\mathbb{C})) \rightarrow (-\infty,\infty)$ and it is jointly continuous at every $(b,D)$ with $b\in\mathbb{R}\backslash\mathbb{Q}$\cite{BJ}\cite{JKS}\cite{JSC}. Given any analytic cocycle $(b,D)$, we consider its holomorphic extension $(b,D_{\varepsilon})$ with $|\varepsilon|\le\eta$.

The Lyapunov exponent $L(b,D_{\varepsilon})$ is easily seen to be a convex function of $\varepsilon$. Thus we can introduce the acceleration of $(b,D_{\varepsilon})$.

\begin{equation}\label{acceleration}\omega(b,D_{\varepsilon})=\frac{1}{2\pi}\mathop{lim}_{h\rightarrow0^{+}}\frac{L(b,D_{\varepsilon+h})-L(b,D_{\varepsilon})}{h}.
\end{equation}

It follows from convexity and continuity of the Lyapunov exponent that the acceleration is an upper semicontinuous function in parameter $\varepsilon$.

\begin{definition}\label{definition2.2}\quad $(b,A)\in(\mathbb{R}\backslash\mathbb{Q})\times C^{\omega}(\mathbb{R}/\mathbb{Z},SL(2,\mathbb{C}))$ is regular if $L(b,A_{\varepsilon})$ is affine for $\varepsilon$ in a neighborhood of 0.
\end{definition}

\begin{remark}\label{remark2.3}\quad If $A$ takes values in $SL(2,\mathbb{R})$, then $\varepsilon\mapsto L(b,A_{\varepsilon})$ is an even function. By convexity, $\omega(b,A)\ge0$. And if $b\in\mathbb{R}\backslash\mathbb{Q}$, then $(b,A)$ is regular if and only if $\omega(b,A)=0$.
\end{remark}

The acceleration was first introduced and the above results were proved in \cite{A} for analytic $SL(2,\mathbb{C})-cocyles$. It was extened to the gengral case $M_{2}(\mathbb{C})$ in \cite{JSC}. 

The acceleration satisfies the following theorem:
\begin{theorem}[Quantization of acceleration \cite{A}\cite{AJS}\cite{JS}]\label{theorem2.4}\quad Consider cocycle $(b,D)$ with $detD(x)$ bound away from 0 on the strip $\mathbb{T}_{\varepsilon}=\{z:|\Im z|<\varepsilon \}$, then $\omega(b,D_{\varepsilon})\in \displaystyle\frac{1}{2}Z$. Morveover, $\omega(b,D_{\varepsilon})\in Z$ for $SL(2,\mathbb{C})-cocycles$.\end{theorem}

\begin{lemma}\label{lemma2.5}\cite{A}\quad $E\notin\sum\nolimits_{b,v}$ if and only if $(b,A)$ is uniformly hyperbolic.\end{lemma}

\begin{lemma}\label{lemma2.6}\cite{A}\quad If $L(b,A)>0$, then $(b,A)$ is regular if and only if $(b,A)$ is uniformly hyperbolic.\end{lemma}

\begin{lemma}\label{lemma2.7}\cite{A}\quad If $(b,A)\in (\mathbb{R}\backslash\mathbb{Q})\times (C^{\omega}(\mathbb{R}/\mathbb{Z}),SL(2,\mathbb{R}))$, then $(b,A)$ is regular if and only if $\omega(b,A)=0$.\end{lemma}
\begin{remark}\label{remark2.8}
	The above lemmas suggest that the accelaration $\omega(b,A)$ is positive for every energy $E$ in the spectrum if the corresponding Lyapunov exponent $L(b,A)$ is positive.
	\end{remark}

Next, we use Avila's global theroy to obtain the formula of Lyapunov exponent in the spectrum. The key point is to prove that accelaration $\omega(b,B_{\varepsilon})=1$ for $\varepsilon\ge0$ in the spectrum.
\begin{proof}[PROOF OF THEOREM \ref{theorem2.8}]\quad
Similarly, denote
	$$B_{\varepsilon}(\theta,E)=(2-\alpha e^{2\pi(\theta+\varepsilon i)i}-\alpha e^{-2\pi(\theta+\varepsilon i)i})A_{\varepsilon}(\theta,E),\ E\in\mathbb{R}.$$
Recalling the definition of acceleration (\ref{acceleration}) and equation (\ref{equation2.5}), we have
	\begin{equation}\label{equation2.7}\omega(b,A_{\varepsilon})=\omega(b,B_{\varepsilon}),\quad |\varepsilon|<\frac{1}{2\pi}log|\frac{1+\sqrt{1-\alpha^2}}{\alpha}|.\end{equation}
If $\varepsilon$ is sufficiently large, it holds that
	$$B_{\varepsilon}(\theta,E)=e^{2\pi\varepsilon-2\pi\theta i}\begin{pmatrix}-\alpha E-2\lambda&\alpha\\-\alpha&0\end{pmatrix}+o(1).$$
Putting the above into the definition of the Lyapunov exponent yields
	$$L(b,B_{\varepsilon})=max\{log|\frac{\alpha E+2\lambda\pm\sqrt{(\alpha E+2\lambda)^2-4\alpha^2}}{2}|\}+2\pi\varepsilon+o(1),$$
	if $\varepsilon$ is sufficiently large.
	
Applying Theorem \ref{theorem2.4}, one obtains that $$\omega(b,B_{\varepsilon})=1,$$ 
and
	\begin{equation}\label{equation2.8}L(b,B_{\varepsilon})=max\{log|\frac{\alpha E+2\lambda\pm\sqrt{(\alpha E+2\lambda)^2-4\alpha^2}}{2}|\}+2\pi\varepsilon,
	\end{equation}
	if $\varepsilon$ is sufficiently large.
	
Since $L(b,B_{\varepsilon})$ is a convex function of $\varepsilon$, we combine Remark \ref{remark2.3} and equation (\ref{equation2.8}) to see that
	\begin{equation}\label{equation2.14}\omega(b,A_{\varepsilon})=\omega(b,B_{\varepsilon})=0\ or\ 1,\quad 0\le\varepsilon<\frac{1}{2\pi}log|\frac{1+\sqrt{1-\alpha^2}}{\alpha}|.
	\end{equation}
	
If $L(b,A)>0$ and $E\in\sum\nolimits_{b,v}$, we employ Remark \ref{remark2.8} and equation \ref{equation2.14} to obtain
	$$\omega(b,A_{\varepsilon})=\omega(b,B_{\varepsilon})=1,\ 0\le\varepsilon<\frac{1}{2\pi}log|\frac{1+\sqrt{1-\alpha^2}}{\alpha}|.$$
Thanks to equation (\ref{equation2.8}) and the continuity of the Lyapunov exponent, one has
$$\omega(b,B_{\varepsilon})=1,\ \varepsilon\ge0.$$
Obviously, there holds
	$$L(b,B_{\varepsilon})=max\{log|\frac{\alpha E+2\lambda\pm\sqrt{(\alpha E+2\lambda)^2-4\alpha^2}}{2}|\}+2\pi\varepsilon,\ \varepsilon\ge0.$$
According to equation (\ref{equation2.5}) and the non-negativity of the Lyapunov exponent, we have
	\begin{equation}\label{equation2.15}
		\begin{split}
			L(b,A)&=L(b,B)-log|1+\sqrt{1+\alpha^2}|\\
			&=max\{log|\frac{\alpha E+2\lambda\pm\sqrt{(\alpha E+2\lambda)^2-4\alpha^2}}{2(1+\sqrt{1-\alpha^2})}|,0\}.
		\end{split}
	\end{equation}
	
If $L(b,A)=0$ and $E\in\sum\nolimits_{b,v}$, it also satifies the above equation (\ref{equation2.15}) by Theorem \ref{theorem2.1}.
\end{proof}

\begin{corollary}\cite{BJ}\label{corollary2.9}\quad
	Let $\alpha=0$, then the model (\ref{mymodel}) becomes the almost Mathieu operator, and $L(E)=max\{log|\lambda|,0\}$ for every $E\in\sum\nolimits_{b,v}$.
\end{corollary}

\begin{corollary}\label{corollary2.10}\quad
Let $v(\theta)=2\lambda\frac{1-cos(2\pi\theta)}{1-\alpha cos(2\pi\theta)}$, then $$L(E)=max\{log|\frac{\alpha E-2\lambda\pm\sqrt{(\alpha E-2\lambda)^2-4\alpha^2}}{2(1+\sqrt{1-\alpha^2})}|,0\},\ \forall\, E\in\sum\nolimits_{b,v}.$$ 
If $\alpha=1$, then it becomes the periodic model and the Lyapunov exponent is $L(E)=0$ for every $E\in\sum\nolimits_{b,v}$; if $\alpha=-1$, then the potential is $v(\theta)=2\lambda\mathop{tan}^2(\pi\theta)$ and the Lyapunov exponent is $$L(E)=max\{log|\frac{E+2\lambda\pm\sqrt{(E+2\lambda)^2-4}}{2}|\}$$ for every $E\in\sum\nolimits_{b,v}.$
\end{corollary}
\begin{proof}\quad If $\alpha=1$, then $v(\theta)=2\lambda$, the spectrum is contained in $[-2+2\lambda,2+2\lambda]$ and $|E-2\lambda|\le2$. Obviously, there holds $L(E)=0$ for every $E\in\sum\nolimits_{b,v}$.

Note that
$$2\lambda\frac{1-cos(2\pi\theta)}{1-\alpha cos(2\pi\theta)}=2\lambda+2\lambda(\alpha-1)\frac{cos2\pi\theta}{1-\alpha cos2\pi\theta},$$ so the new coupling is $\lambda(\alpha-1)$, the new energy is $E-2\lambda$. The corollary follows by Theorem \ref{theorem2.8}.
\end{proof}
\section{Coexistence of zero Lyapunov exponent and positive Lyapunov exponent}
In this section, we use the formula of the Lyapunov exponent in Theorem \ref{theorem2.8} to study the coexistence of zero Lyapunov exponent and positive Lyapunov exponent.

By operator theory, it is easy to see that $\sum\nolimits_{b,v}\subseteq [-2+min(v),2+max(v)]$, and we have the following lemma.

\begin{lemma}\label{lemma3.1}\quad 
If $\lambda>0$, then $\sum\nolimits_{b,v}\subseteq [-2-\frac{2\lambda}{1+\alpha},2+\frac{2\lambda}{1-\alpha}]$. If $\lambda<0$, then $\sum\nolimits_{b,v}\subseteq [-2+\frac{2\lambda}{1-\alpha},2-\frac{2\lambda}{1+\alpha}]$. If $\lambda=0$, then $\sum\nolimits_{b,v}=[-2,2]$.
\end{lemma}

\begin{proof}\quad 
Just consider the fact that $\sum\nolimits_{b,v}\subseteq [-2+min(v),2+max(v)]$. When $\lambda>0$, $min(v)=-\frac{2\lambda}{1+\alpha}$, $max(v)=\frac{2\lambda}{1-\alpha}$; when $\lambda<0$, $min(v)=\frac{2\lambda}{1-\alpha}$, $max(v)=-\frac{2\lambda}{1+\alpha}$; when $\lambda=0$, it is a periodic model and $\sum\nolimits_{b,v}=[-2,2]$.
\end{proof}

Then we use the formula of Lyapunov exponent and Lemma \ref{lemma3.1} to prove the Theorem \ref{theorem3.2}.
\begin{proof}[PROOF OF THEOREM \ref{theorem3.2}]\quad 
For every $E$ in $\sum\nolimits_{b,v}$, according to the formula of the Lyapunov exponent \ref{theorem2.8}, we have
$$L(E)=max\{log\frac{|\alpha|}{1+\sqrt{1-\alpha^2}},0\}=0,\ if |\alpha E+2\lambda|\le2|\alpha|.$$
Hence it holds that
$$\begin{matrix*}[l]L(E)>0&\Longleftrightarrow\left\{\begin{matrix*}[l]\alpha E+2\lambda+\sqrt{(\alpha E+2\lambda)^2-4\alpha^2}>2(1+\sqrt{1-\alpha^2})&\alpha E+2\lambda>2|\alpha|,\\ \alpha E+2\lambda-\sqrt{(\alpha E+2\lambda)^2-4\alpha^2}<-2(1+\sqrt{1-\alpha^2})&\alpha E+2\lambda<-2|\alpha|,\end{matrix*}\right.\\
&\Longleftrightarrow|\alpha E+2\lambda|>2.\end{matrix*}$$

Assume $\alpha E+2\lambda<-2$ when $\lambda>0$, we have
	$$\sum\nolimits_{b,v}\subseteq[-2-\frac{2\lambda}{1+\alpha},2+\frac{2\lambda}{1-\alpha}],$$
and
	$$\left\{\begin{matrix*}[l]E<\frac{-2-2\lambda}{\alpha}<-2-\frac{2\lambda}{1+\alpha}&\alpha>0,\\E>\frac{-2-2\lambda}{\alpha}>2+\frac{2\lambda}{1-\alpha}&\alpha<0.\end{matrix*}\right.$$
This contradiction means that $\alpha E+2\lambda>2$ when $\lambda>0$.

Assume $\alpha E+2\lambda>2$ when $\lambda<0$, we have
	$$\sum\nolimits_{b,v}\subseteq[-2+\frac{2\lambda}{1-\alpha},2-\frac{2\lambda}{1+\alpha}],$$
and
	$$\left\{\begin{matrix*}[l]E>\frac{2-2\lambda}{\alpha}>2-\frac{2\lambda}{1+\alpha}&\alpha>0,\\E<\frac{2-2\lambda}{\alpha}<-2+\frac{2\lambda}{1-\alpha}&\alpha<0.\end{matrix*}\right.$$
This contradiction means that $\alpha E+2\lambda<-2$ when $\lambda<0$.

It is known that $L(E)=0$ for every $E\in\sum\nolimits_{b,v}$ when $\lambda=0$.

In conclusion, for every $E\in\sum\nolimits_{b,v}$ the Lyapunov exponent $L(E)$ is positive if and only if
	$$\left\{\begin{matrix*}[l]\alpha E>2sgn(\lambda)(1-|\lambda|)&\lambda>0,\\\alpha E<2sgn(\lambda)(1-|\lambda|)&\lambda<0,\end{matrix*}\right.$$
for every $E\in\sum\nolimits_{b,v}$ the Lyapunov exponent $L(E)$ is zero if and only if
	$$\left\{\begin{matrix*}[l]\alpha E\le2sgn(\lambda)(1-|\lambda|)&\lambda>0,\\\alpha E\ge2sgn(\lambda)(1-|\lambda|)&\lambda<0,\\ \forall\, E\in\sum\nolimits_{b,v}&\lambda=0.\end{matrix*}\right.
	$$
The boundary is 
	\begin{equation*}
	\alpha E=2sgn(\lambda)(1-|\lambda|).
	\end{equation*}
	
\end{proof}

Furthermore, it is easy to prove Theorem \ref{theorem3.3} by Theorem \ref{theorem3.2} and Lemma \ref{lemma3.1}
\begin{proof}[PROOF OF THEOREM \ref{theorem3.3}]\quad 
For simplicity, we only prove the theorem in the case of $\lambda>0$ and $\alpha>0$, the proof is the same in the other cases.

Assume $\{ E:L(E)=0,E\in\sum\nolimits_{b,v}\}=\emptyset$ when $\lambda>0$, $\alpha>0$, $\lambda<1+\alpha$. By Theorem \ref{theorem3.2}, $E>\frac{2(1-\lambda)}{\alpha}$ for every $E\in\sum\nolimits_{b,v}$. Select $n$ such that $cos(2\pi(nb+\theta))<\frac{1-\lambda}{\alpha}$ ($n$ exists because of $\frac{1-\lambda}{\alpha}>-1$), we have
$$\frac{2(1-\lambda)}{\alpha}>\left <\delta_{n},H_{\alpha,\theta,b}\delta_{n}\right >=\frac{2\lambda cos(2\pi(nb+\theta))}{1-\alpha cos(2\pi(nb+\theta))}=\int_{\mathbb{R}}Ed\mu_{\delta_{n}}>\frac{2(1-\lambda)}{\alpha},$$
$\mu_{\delta_{n}}$ is the spectral measure on $\sigma(H_{\alpha,\theta,b})$, see the definitions (\ref{measure}).
This contradiction means $\{E:L(E)=0,E\in\sum\nolimits_{b,v}\}\neq\emptyset$ when $\lambda>0$, $\alpha>0$, $\lambda>1-\alpha$.

Assume $\{ E:L(E)>0,E\in\sum\nolimits_{b,v}\}=\emptyset$ when $\lambda>0$, $\alpha>0$, $\lambda>1-\alpha$. By Theorem \ref{theorem3.2}, $E\le\frac{2(1-\lambda)}{\alpha}$ for every $E\in\sum\nolimits_{b,v}$. Select $n$ such that $cos(2\pi(nb+\theta))>\frac{1-\lambda}{\alpha}$ ($n$ exists because of $\frac{1-\lambda}{\alpha}<1$), we have
$$\frac{2(1-\lambda)}{\alpha}<\left <\delta_{n},H_{\alpha,\theta,b}\delta_{n}\right >=\frac{2\lambda cos(2\pi(nb+\theta))}{1-\alpha cos(2\pi(nb+\theta))}=\int_{\mathbb{R}}Ed\mu_{\delta_{n}}\le\frac{2(1-\lambda)}{\alpha}.$$
This contradiction means $\{E:L(E)>0,E\in\sum\nolimits_{b,v}\}\neq\emptyset$ when $\lambda>0$, $\alpha>0$, $\lambda<1+\alpha$.
The above discussions imply that there are both zero Lyapunov exponent and positive Lyapunov exponent in the spectrum when $\lambda>0$, $\alpha>0$, $1-\alpha<\lambda<1+\alpha$.

According to Lemma \ref{lemma3.1}, the spectrum $\sum\nolimits_{b,v}$ is contained in $[-2-\frac{2\lambda}{1+\alpha},2+\frac{2\lambda}{1-\alpha}]$. By Theorem \ref{theorem3.2}, there is no positive Lyapunov exponent in the spectrum if $\frac{2(1-\lambda)}{\alpha}\ge2+\frac{2\lambda}{1-\alpha}$, i.e, $\lambda\le(1-\alpha)^2$. There is no zero Lyapunov exponent in the spectrum if $\frac{2(1-\lambda)}{\alpha}<-2-\frac{2\lambda}{1+\alpha}$, i.e, $\lambda>(1+\alpha)^2$.

Repetition of the same arguments leads to
	\begin{equation*}
	\left\{\begin{matrix*}[l]\exists\,E_{1},E_{2}\in\sum\nolimits_{b,v},\ L(E_{1})>0,L(E_2)=0&if\ 1-|\alpha|<|\lambda|<1+|\alpha|,\\
	\forall\,E\in\sum\nolimits_{b,v},\ L(E)>0&if\ |\lambda|>(1+|\alpha|)^2,\\
	\forall\,E\in\sum\nolimits_{b,v},\ L(E)=0&if\ |\lambda|\le(1-|\alpha|)^2.\end{matrix*}\right.
	\end{equation*}
	\end{proof}
In terms of Theorem \ref{theorem3.3}, we can obtain the Corollary \ref{corollary1.6}.
\begin{proof}[PROOF OF COROLLARY \ref{corollary1.6}]\quad Notice that
$$v(\theta)=2\lambda tan^2\pi\theta=2\lambda\frac{1-cos2\pi\theta}{1+cos2\pi\theta)}=2\lambda-4\lambda\frac{cos2\pi\theta}{1+cos2\pi\theta},$$ so the new coupling is $-2\lambda$, the new energy is $E-2\lambda$, the parameter is $-1$. By Theorem \ref{theorem3.3}, there are both zero Lyapunov exponent and positive Lyapunov exponent in the spectrum if $0<|\lambda|<1$.
\end{proof}

\begin{remark}\quad
The potential is infinite and non-analytic in $\mathbb{T}$, and the Schr$\ddot{o}$dinger equation has zero Lyapunov exponent in some energies. It is very different from the Maryland model($v(\theta)=2\lambda tan(\pi\theta)$) which has positive Lyapunov exponent if $|\lambda|\neq0$. 
\end{remark}

\section{Anderson localization and proof of Problem 1.1.}
In this section, we use similar methods in \cite{J} to prove that the spectrum in the region of positive Lyapunov exponent is purely pure point spectrum with exponentially decaying eigenfunctions (Anderson localization).

To study the spectral properties of the operator $H_{\alpha,\theta,b}$, we have to introduce its universal spectral measure $\mu_{\alpha,\theta,b}$ on $\sigma(H_{\alpha,\theta,b})$ defined as follows: 
\begin{equation}
\mu_{\alpha,\theta,b}=\frac{1}{2}(\mu_{\alpha,\theta,b,\delta_{0}}+\mu_{\alpha,\theta,b,\delta_{1}}),\ \delta_{i}(n)=\delta_{i,n},\ i=0,1
\end{equation}
where the spectral measures $\mu_{\delta,\theta,b,\delta_{i}},\ i=0,1$ are uniquely defined by
\begin{equation}\label{measure}
\left <\delta_{i},(H_{\delta,\theta,b}-zI)^{-1}\delta_{i}\right >=\int_{\sigma(H_{\alpha,\theta,b})}\frac{d\mu_{\alpha,\theta,b,\delta_{i}}(t)}{t-z},\ \forall\,z \in \mathbb{C},\ \Im z>0.
\end{equation}
They are probability measures.

The universal spectral measure $\mu_{\alpha,\theta,b}$ has a unique Lebesgue decomposition:
\begin{equation*}
\mu_{\alpha,\theta,b}=\mu_{\alpha,\theta,b,pp}+\mu_{\alpha,\theta,b,sc}+\mu_{\alpha,\theta,b,ac},
\end{equation*}
where $\mu_{\alpha,\theta,b,pp}$ is a pure  point measure, $\mu_{\alpha,\theta,b,sc}$ is a singular continuous measure, and $\mu_{\alpha,\theta,b,ac}$ is an absolutely continuous measure. This means that $\mu_{\alpha,\theta,b,pp}(\mathbb{R}\backslash C)=0$ for some countable set $C$, $\mu_{\alpha,\theta,b,sc}(\{E\})=0$ for every $E\in\mathbb{R}$, $\mu_{\alpha,\theta,b,sc}(\mathbb{R}\backslash N)=0$ for some set $N$ of zero Lebesgue measure, and $\mu_{\alpha,\theta,b,ac}(N)=0$ for every set $N$ of zero Lebesgue measure. We also denote $\mu_{\alpha,\theta,b,c}=\mu_{\alpha,\theta,b,sc}+\mu_{\alpha,\theta,b,ac}$ for the continuous part and $\mu_{\alpha,\theta,b,s}=\mu_{\alpha,\theta,b,pp}+\mu_{\alpha,\theta,b,sc}$ for the singular part of $\mu_{\alpha,\theta,b}$. 

Given this measure decomposition, we can define the following subsets of $l^2(\mathbb{Z})$:
\begin{equation*}
\begin{split}
l^2(\mathbb{Z})_{ac}&=\{\psi\in l^2(\mathbb{Z}):\mu_{\alpha,\theta,b}=\mu_{\alpha,\theta,b,ac}\},\\
l^2(\mathbb{Z})_{sc}&=\{\psi\in l^2(\mathbb{Z}):\mu_{\alpha,\theta,b}=\mu_{\alpha,\theta,b,sc}\},\\
l^2(\mathbb{Z})_{pp}&=\{\psi\in l^2(\mathbb{Z}):\mu_{\alpha,\theta,b}=\mu_{\alpha,\theta,b,pp}\}.
\end{split}
\end{equation*}

Each of these subsets turns out to be a closed subspace, and
$$l^2(\mathbb{Z})=l^2(\mathbb{Z})_{ac}\oplus l^2(\mathbb{Z})_{sc}\oplus l^2(\mathbb{Z})_{pp}.$$

We also considers the continuous subspace:
$$l^2(\mathbb{Z})_{c}=l^2(\mathbb{Z})_{ac}\oplus l^2(\mathbb{Z})_{sc}.$$

The spectrum of the restriction of $H_{\alpha,\theta,b}$ to $l^2(\mathbb{Z})_{ac}$ is denoted by $\sigma_{ac}(H_{\alpha,\theta,b})$ and called the absolutely continuous spectrum of $H_{\alpha,\theta,b}$. The sets $\sigma_{sc}(H_{\alpha,\theta,b})$,  $\sigma_{pp}(H_{\alpha,\theta,b})$,  $\sigma_{c}(H_{\alpha,\theta,b})$ are defined similarly and called singular continuous, pure point and continuous spectrum of $H_{\alpha,\theta,b}$, respectively. We have $\sigma(H_{\alpha,\theta,b})=supp\;\mu_{\alpha,\theta,b}$, $\sigma_{ac}=supp\;\mu_{\alpha,\theta,b,ac}$, $\sigma(H_{\alpha,\theta,b,sc})=supp\;\mu_{\alpha,\theta,b,sc}$ and $\sigma(H_{\alpha,\theta,b,pp})=supp\;\mu_{\alpha,\theta,b,pp}$.

For $[n_{1},n_{2}]={n\in\mathbb{Z}:n_{1}\le n\le n_{2}}$, denote by $H_{[n_{1},n_{2}]}$ the restriction of $H$ to this interval with zero boundary conditions at $n_{1}-1$ and $n_{2}+1$: that is, $H_{[n_{1},n_{2}]}=P_{[n_{1},n_{2}]}HP_{[n_{1},n_{2}]}^{*}$ where $P_{[n_{1},n_{2}]}:l^2(\mathbb{Z})\to l^([n_{1},n_{2}])$ is the canonical projection, and $P_{[n_{1},n_{2}]}^{*}:l^2([n_{1},n_{2}])\to l^2(\mathbb{Z})$ is the canonical embedding.

Moreover, for $E\not\in\sigma(H_{[n_{1},n_{2}]})$ and $n,m\in[n_{1},n_{2}]$, let
$$G_{[n_{1},n_{2}]}(n,m;E)=\left <\delta_{n},(H_{[n_{1},n_{2}]}-E)^{-1}\delta_{m}\right >,$$
it is called Green function.

We say that $E\in\mathbb{R}$ is a generalized eigenvalue if equation (\ref{equation2.1}) has a non-trival solution $u_{E}$, called the corresponding generalized eigenfunction, satisfying
$$|u_{E}|\le C(1+|n|)^{\delta}$$
for suitable finite contants $C$ and $\delta$, and every $n\in\mathbb{Z}$.

\begin{definition}\label{definition3.5}\quad 
An irrational number $b\in\mathbb{T}$ is called Diophantine if there are constants $c=c(b)>0$ and $r=r(b)>1$ such that
	$$|sin(2\pi nb)|>\frac{c}{|n|^r}\ for\ every\ n\in\mathbb{Z}\backslash\{0\},$$
	and $\theta\in\mathbb{T}$ is called resonant respect to above number $b$ if the relation
	$$|sin(2\pi(\theta+\frac{n}{2}b))|<e^{-|n|^{\frac{1}{2r}}}$$
holds for infinitely many $n\in\mathbb{Z}$; otherwise $\theta$ is called non-resonant.
\end{definition}
\begin{remark}\label{remark3.6}\quad 
Lebesgue almost every $b\in\mathbb{T}$ is Diophantine, and the set of resonant $\theta$'s is a dense $G_{\delta}$ set (as can be seen directly from the definition) of zero Lebesgue measure (by Borel-Cantelli) so almost every $\theta\in\mathbb{T}$ is non-resonant respect to $b$.
\end{remark}

\begin{definition}\label{definition3.7}\quad 
Let $$P_{k}(\theta,E)=det[|(H_{\alpha,\theta,b}-E)|_{[0,k-1]}]$$
and
	$$\mathcal{K}=\{k\in\mathbb{Z}_{+} :\exists\,\theta\in\mathbb{T}\ with\ |P_{k}(\theta,E)|\ge\frac{1}{\sqrt{2}}e^{kL(E)}\}.$$
\end{definition}
In fact, 
\begin{equation*}
det[|(H_{\alpha,\theta,b}-E)|_{[0,k-1]}]=det\begin{pmatrix}v_{\theta}(0)-E&1& & & \\1&v_{\theta}(1)-E&1& & \\ &1&\ddots&\ddots& \\ & &\ddots&\ddots&1\\ & & &1&v_{\theta}(k-1)-E\end{pmatrix},
\end{equation*}
where $v_{\theta}(j)=v(jb+\theta),\ j=0,\dots,k-1.$
\begin{lemma}\label{lemma3.8}\quad 
There are coefficients $b_{j}$, $0\le j \le k$, such that
	$$P_{k}(\theta,E)=\frac{\sum\limits_{j=0}^{k}b_{j}(cos(2\pi(\theta+\frac{k-1}{2}b)))^{j}}{\prod\limits_{j=0}^{k-1}(1-\alpha cos(2\pi (\theta+jb)))}.$$
\end{lemma}

\begin{proof}\quad 
Since cos is an even function, denote $U$ the change of basis $\delta_{j}\mapsto\delta_{k-1-j}$, then
	$$U^{-1}H_{\alpha,\theta-\frac{k-1}{2}b,b}|_{[0,k-1]}U=H_{\alpha,-\theta-\frac{k-1}{2}b,b}|_{[0,k-1]}.$$
Thus,
	$$P_{k}(\theta-\frac{k-1}{2}b,E)=P_{k}(-\theta-\frac{k-1}{2}b,E).$$
Denote 
	$$Q_{k}(\theta,E)=P_{k}(\theta,E)\prod\limits_{j=0}^{k-1}(1-\alpha cos(2\pi (\theta+jb)))$$
Due to
	$$\prod\limits_{j=0}^{k-1}(1-\alpha cos(2\pi (\theta-\frac{k-1}{2}b+jb)))=\prod\limits_{j=0}^{k-1}(1-\alpha cos(2\pi (-\theta-\frac{k-1}{2}b+jb))).$$
we obtain
	\begin{equation}\label{equation3.6}
	Q_{k}(\theta-\frac{k-1}{2}b,E)=Q_{k}(-\theta-\frac{k-1}{2}b,E).
	\end{equation}
Thus, the Fourier expansion of $\theta\mapsto Q_{k}(\theta-\frac{k-1}{2}b,E)$ reads
	$$Q_{k}(\theta,E)=\sum\limits_{j=0}^{k}a_{j}cos(2\pi j(\theta+\frac{k-1}{2}b))$$
since all the sin terms are absent due to (\ref{equation3.6}) and the degree obviously does not exceed k. The lemma follows since the linear span of $\{1,cos(2\pi x),$ $cos(2\pi 2x),$ $\dots,cos(2\pi kx)\}$ is equal to that of $\{1,cos(2\pi x),$ $cos^{2}(2\pi x),$ $\dots,cos^{k}(2\pi x)\}$.
\end{proof}

\begin{theorem}[Kingman 1973]\label{theorem3.9}\quad 
Suppose $(\Omega,\mu,T)$ is ergodic. If $f_{n}:\Omega\rightarrow\mathbb{R}$ are measurable, obey $\parallel f_{n}\parallel_{\infty}\lesssim n$ and the subadditivity condition
	$$f_{n+m}(\theta)\le f_{n}(\theta)+f_{m}(T^{n}\theta),$$
then
	$$\mathop{lim}_{n\to\infty}\frac{1}{n}f_{n}(\theta)=\mathop{inf}_{n\ge1}\frac{1}{n}\mathbb{E}(f_{n})$$
for $\mu$-almost every $\theta\in\Omega$.
\end{theorem}

\begin{lemma}\label{lemma3.10}\quad 
For every $k\in\mathbb{Z}_{+}$, at least one of $k,k+1,k+2$ belongs to $\mathcal{K}$.
\end{lemma}

\begin{proof}\quad 
Recall that the transfer matrix $M_{E}(k,\theta)$ may be written as 
	\begin{equation}\label{equation3.7}
	A_{n}(\theta,E)=\begin{pmatrix}P_{k}(\theta,E)&-P_{k-1}(\theta+b,E)\\P_{k-1}(\theta,E)&-P_{k-2}(\theta+b,E)\end{pmatrix}.
	\end{equation}
Therefore, the statement of lemma follows from Kingman's subadditive ergodic Theorem \ref{theorem3.9}.
\end{proof}

When the Lyapunov exponent is positive, on average the transfer matrices have exponentially large norm, and hence some of the entries must be exponentially large. These entries in turn appear in a description of the Green's function of the operator restricted to a finite interval. Namely, by Cramer's Rule, we have for $n_{1},n_{2}=n_{1}+k-1$, and $n\in [n_{1},n_{2}]$,
\begin{equation}\label{equation3.8}
|G_{[n_{1},n_{2}]}(n_{1},n;E)|=|\frac{P_{n_{2}-n}(\theta+(n+1)b,E)}{P_{k}(\theta+n_{1}b,E)}|,
\end{equation}

\begin{equation}\label{equation3.9}
|G_{[n_{1},n_{2}]}(n,n_{2};E)|=|\frac{P_{n-n_{1}}(\theta+n_{1}b,E)}{P_{k}(\theta+n_{1}b,E)}|.
\end{equation}

\begin{theorem}[Furman 1997]\label{theorem3.11}\quad 
Suppose $(\Omega,T)$ is uniquely ergodic. If $f_{n}:\Omega\to\mathbb{R}$ are continuous and obey the subadditivity condition  $f_{n+m}(\theta)\le f_{n}(\theta)+f_{m}(T^{n}\theta)$, then
	$$\mathop{lim\ sup}_{n\to\infty}\frac{1}{n}f_{n}(\theta)\le \mathop{inf}_{n\ge1}\frac{1}{n}\mathbb{E}(f_{n})$$
for every $\theta\in\Omega$ and uniformly on $\Omega$.
\end{theorem}

\begin{lemma}\label{lemma3.12}\quad 
For every $E\in\mathbb{R}$ and $\varepsilon>0$, there exists $k(E,\varepsilon)$ such that 
	$$|P_{k}(\theta,E)|<e^{(L(E)+\varepsilon)k}$$
for every $k>k(E,\varepsilon)$ and every $\theta\in\mathbb{T}$.
\end{lemma}
\begin{proof}\quad 
It is a consequence of equation (\ref{equation3.7}) and Theorem \ref{theorem3.11}.
\end{proof}

\begin{definition}\label{definition3.13}\quad 
Fix $E\in\mathbb{R}$ and $\gamma\in\mathbb{R}$. A point $n\in\mathbb{Z}$ will be called $(\gamma,k)$-regular if there exists an interval $[n_{1},n_{2}]$, containing $n$ such that
	\begin{equation*}
	\begin{split}
	(i)&\quad n_{2}=n{1}+k-1,\\
	(ii)&\quad n\in[n_{1},n_{2}],\\
	(iii)&\quad |n-n_{i}|>\frac{k}{5},\\
	(iv)&\quad |G_{[n_{1},n_{2}]}(n,n_{i};E)|<e^{-\gamma |n-n_{i}|}.
	\end{split}
	\end{equation*}
Otherwise, $n$ is called $(\gamma,k)$-singular.
	
\end{definition}

\begin{lemma}\label{lemma3.14}\quad 
Fix $E\in\mathbb{R}$. Suppose $n$ is $(L(E)-\varepsilon,k)-singular$ for some $0<\varepsilon<\frac{L(E)}{3}$ and $k>4k(E,\frac{\varepsilon}{6})+1$. Then, for every $j$ with
	$$n-\frac{3}{4}k\le j\le n-\frac{3}{4}k+\frac{k+1}{2},$$
we have that
	$$|P_{k}(\theta+jb,E)|\le e^{k(L(E)-\frac{\varepsilon}{8})}.$$
\end{lemma}

\begin{proof}\quad 
Since $n$ is $(L(E)-\varepsilon,k)-singular$, it follows that for every interval
	$[n_{1},n_{2}]$ of length $k$ containing $n$ with $|n-n_{i}|>\frac{k}{5}$, we have that 
	$$|G_{[n_{1},n_{2}]}(n,n_{i};E)|\ge e^{-(L(E)-\varepsilon)|n-n_{i}|}.$$
By equation (\ref{equation3.8}), this means that 
	$$|\frac{P_{n_{2}-n}(\theta+(n+1)b,E)}{P_{k}(\theta+n_{1}b,E)}|\ge e^{-(L(E)-\varepsilon)|n-n_{i}|}.$$
We can choose $n_{1}$ to be equal to the $j$ in question and set $n_{2}=j+k-1$. Then we find, using Lemma \ref{lemma3.12},
	\begin{equation*}
	\begin{split}
	|P_{k}(\theta+jb,E)|&\le |P_{j+k-1-n}(\theta+(n+1)b,E)|e^{(L(E)-\varepsilon)|n-j|} \\ 
	&\le e^{(L(E)+\frac{\varepsilon}{6})|j+k-1-n|}e^{(L(E)-\varepsilon)|n-j|} \\
	&= e^{(k-1)L(E)+\varepsilon(\frac{j+k-1-n}{6}+j-n)} \\
	&\le e^{(k-1)L(E)+\varepsilon(\frac{5}{12}-\frac{1}{8}k)} \\
	&<e^{k(L(E)-\frac{\varepsilon}{8})+\frac{5}{36}L(E)-L(E))} \\
	&<e^{k(L(E)-\frac{\varepsilon}{8})}.
	\end{split}
	\end{equation*}
\end{proof}

\begin{lemma}\label{lemma3.15}\cite{J}\quad 
Suppose $n\in [n_{1},n_{2}]\subseteq\mathbb{Z}$ and $u$ is a solution of the equation $Hu=Eu$. Then,
	\begin{equation}\label{equation3.10}
	u(n)=-G_{[n_{1},n_{2}]}(n,n_{1};E)u(n_{1}-1)-G_{[n_{1},n_{2}]}(n,n_{2};E)u(n_{2}+1).
	\end{equation}
In particular, if $u_{E}$ is a generalized eigenfunction, then every point $n\in\mathbb{Z}$ with $u_{E}(n)\neq0$ is $(\gamma,k)$-singular for $k>k_{1}=k_{1}(E,\gamma,\theta,n)$.
\end{lemma}

\begin{lemma}\label{lemma3.16}\quad 
For every $n\in\mathbb{Z}$, $\varepsilon>0$, $\tau <2$, there exists $k_{2}=k_{2}(\theta,b,n,\varepsilon,\tau,E)$ such that for every $k\in\mathcal{K}$ with $k>k_{2}$, we have that
	\begin{center}$m$, $n$ are both $(L(E)-\varepsilon,k)$-singular and $|m-n|>\frac{k+1}{2}\Rightarrow |m-n|>k^{\tau}$.
	\end{center}
\end{lemma}

\begin{proof}\quad 
Assume that $m_{1}$ and $m_{2}$ are both $(L(E)-\varepsilon,k)-singular$ with 
	$$d=m_{2}-m_{1}>\frac{k+1}{2}.$$
Let
	$$n_{i}=m_{i}-\lfloor\frac{3}{4}k\rfloor,\ i=1,2.$$
By Lemma \ref{lemma3.8}, there is a polynomial $R_{k}$ of degree $k$ such that
	$$P_{k}(\theta,E)\prod\limits_{j=0}^{k-1}(1-\alpha cos(2\pi (\theta+jb)))=R_{k}(cos(2\pi (\theta+\frac{k-1}{2}b))).$$
Let
	\begin{equation*}
	\theta_{j}=\left\{
	\begin{matrix*}[l]
	\theta+(n_{1}+\frac{k-1}{2}+j)b,&j=0,1,\dots,\lfloor\frac{k+1}{2}\rfloor-1,\\
	\theta+(n_{2}+\frac{k-1}{2}+j-\lfloor\frac{k+1}{2}\rfloor)b,&j=\lfloor\frac{k+1}{2}\rfloor),\lfloor\frac{k+1}{2}\rfloor)+1,\dots,k.
	\end{matrix*}\right .
	\end{equation*}
The points $\theta_{0},\theta_{1},\dots,\theta_{k}$ are distinct.
Lagrange interpolation then shows 
	$$|R_{k}(z)|=|\sum\limits_{j=0}^{k}R_{k}(cos(2\pi\theta_{j})\frac{\prod\nolimits_{l\neq j}(z-cos(2\pi\theta_{l}))}{\prod\nolimits_{l\neq j}cos(2\pi\theta_{j})-cos(2\pi\theta_{l})})|.$$
Due to
	\begin{center}
		$\displaystyle\frac{1}{k}\sum\limits_{j=0}^{k-1}log(1-\alpha cos(2\pi(\theta+jb)))\to log\frac{1+\sqrt{1-\alpha^2}}{2}$, as $k\to\infty.$
	\end{center}
there exists $k_{3}$ such that for $k>k_{3}$
	$$e^{k(log\frac{1+\sqrt{1-\alpha^2}}{2}-\frac{\varepsilon}{64})}<|\prod\limits_{j=0}^{k-1}(1-\alpha cos(2\pi (\theta+jb)))|<e^{k(log\frac{1+\sqrt{1-\alpha^2}}{2}+\frac{\varepsilon}{64})},$$
and by Lemma \ref{lemma3.12} there exists $k_{4}$ such that for $k>k_{4}$
	$$|P_{k}(cos(2\pi\theta_{j}))|<e^{k(L(E)-\frac{\varepsilon}{8})},\ j=0,1,\dots,k.$$
By Lemma 7 in \cite{J}, we know that if $d<k^{\tau}$ for some $\tau<2$, there exists $k_{5}$ so that for $k>k_{5}$, we have 
	\begin{center}
		\begin{equation}\label{equation3.11}
		\frac{|\prod\nolimits_{l\neq j}(z-cos(2\pi\theta_{l}))|}{|\prod\nolimits_{l\neq j}(cos(2\pi\theta_{j})-cos(2\pi\theta_{l}))|}\le e^{\frac{k\varepsilon}{16}}\ for\ z\in [-1,1],\ 0\le j\le k.
		\end{equation}
		
	\end{center}
Given $\tau<2$, consider $k\in\mathcal{K}$ with $k>max\{k_{3},k_{4},k_{5}\}$ and $\tilde{\theta}$ with 
	$$|P_{k}(\tilde{\theta})|\ge\frac{1}{\sqrt{2}}e^{kL(E)}.$$
But assuming $d<k^{\tau}$, we also have the following upper bound,
	\begin{equation*}
	\begin{split}
	|P_{k}(\tilde{\theta})|&\le e^{-k(log\frac{1+\sqrt{1-\alpha^2}}{2}-\frac{\varepsilon}{64})}(k+1)e^{k(L(E)-\frac{\varepsilon}{8})}e^{k(log\frac{1+\sqrt{1-\alpha^2}}{2}+\frac{\varepsilon}{64})}e^{\frac{k\varepsilon}{16}}\\
	&=(k+1)e^{k(L(E)-\frac{\varepsilon}{32})}.
	\end{split}
	\end{equation*}
This contradiction shows that $d<k^{\tau}$ is impossible.
\end{proof}

Then, we provide the detailed proof of Theorem \ref{theorem3.17}.

\begin{proof}[PROOF OF THEOREM \ref{theorem3.17}]\quad 
Let $E(\theta)$ be a generalized eigenvalue of $H_{\alpha,\theta,b}$, and denote the correspongding generalized eigenfunction by $u_{E}$. Assume without loss of generality $u_{E}(0)\neq 0$(otherwise replace 0 by 1).
By Lemma \ref{lemma3.15} and Lemma \ref{lemma3.16}, if
	$$|n|>max\{k_{1}(E,L(E)-\varepsilon,\theta,0),k_{2}(\theta,b,0,\varepsilon,1.5,E)\}+1,$$
the point $n$ is $(L(E)-\varepsilon,k)$-regular for some $k\in\{ |n|-1,|n|,|n|+1\}\cap\mathcal{K}\neq\emptyset$, since 0 is $(L(E)-\varepsilon,k)$-singular. Thus, there exists an interval $[n_{1},n_{2}]$ of length $k$ containing $n$ such that 
	$$\frac{1}{5}(|n|-1)\le |n-n_{i}|\le\frac{4}{5}(|n|+1),$$
and
	$$|G_{[n_{1},n_{2}]}(n,n_{i})|<e^{-(L(E)-\varepsilon)|n-n_{i}|}.$$
By above inequation and equation (\ref{equation3.10}), we obtain that
	$$|u_{E}(n)|\le2C(2|n|+1)^{\delta}e^{-(\frac{L(E)-\varepsilon}{5})(|n|-1)},$$
where $C$ is a constant.
This implies exponential decay in the region of positive Lyapunov exponent if $\varepsilon$ is chosen small enough.
\end{proof}
We need a theorem to prove Problem \ref{problem1.1}.

\begin{theorem}[Ruelle 1979]\label{theorem4.13}\quad
Suppose $A_{n}\in SL(2,\mathbb{C})$ obey
$$\mathop{lim}\limits_{n\rightarrow\infty}\frac{1}{n}||A_{n}||=0$$
and
$$\mathop{lim}\limits_{n\rightarrow\infty}||A_{n}\dots A_{1}||=\gamma>0.$$
Then there exists a one-dimensional subspace $V\subseteq\mathbb{R}^2$ such that
$$\mathop{lim}\limits_{n\rightarrow\infty}||A_{n}\dots A_{1}||=-\gamma\quad for\ v\in V\backslash\{0\}$$
and
$$\mathop{lim}\limits_{n\rightarrow\infty}||A_{n}\dots A_{1}||=\gamma\quad for\ v\notin V.$$
\end{theorem}
Then, it is easy to prove the problem \ref{problem1.1}.
\begin{proof}[PROOF OF PROBLEM \ref{problem1.1}]\quad
By Theorem \ref{theorem4.13}, localized states only appear in the region of positive Lyapunov exponent. Therefor, Problem \ref{problem1.1} is the consequence of Theorem \ref{theorem3.2} and Theorem \ref{theorem3.17}.
\end{proof}

\Acknowledgements{The authors would like to thank Qi Zhou for many useful suggestions and help. This work was supported by the NSF of China (No. 11671382), CAS Key Project of Frontier Sciences (N0. QYZDJ-SSW-JSC003), the Key Lab. of Random Complex Structures and Data Sciences CAS and National Center for Mathematics and Interdisciplinary Sciences CAS. }


\end{document}